\documentclass{amsart}
\usepackage{amsmath,amsthm,amssymb}

\makeatletter
    
    \@addtoreset{equation}{section}
  \makeatother

\newtheorem{definition}{Definition}[section]

\newtheorem{theorem}[definition]{Theorem}
\newtheorem{lemma}[definition]{Lemma}
\newtheorem{corollary}[definition]{Corollary}
\newtheorem{remark}{Remark}[section]

\newcommand{\re}{\mathbb R} 

\DeclareMathOperator{\diver}{div}
\DeclareMathOperator{\rot}{rot}

\pagebreak

\title[2D elliptic equations]{
Asymptotic behavior of solutions to 
elliptic and parabolic equations with unbounded coefficients of the second order
in unbounded domains
} 
\dedicatory{Dedicated to Professor  Matthias Hieber on the occasion of his 60th birthday}
\author{Hideo Kozono, Yutaka Terasawa and Yuta Wakasugi}
\address[H. Kozono]{Department of Mathematics, Faculty of Science and Engineering,
Waseda University, Tokyo 169--8555, Japan, 
Research Alliance Center of Mathematical Sciences, Tohoku University, 
Sendai 980-8578, Japan}
\email[H. Kozono]{kozono@waseda.jp, hideo.kozono.c7@tohoku.ac.jp}
\address[Y. Terasawa]{Graduate School of Mathematics, Nagoya University,
Furocho Chikusaku Nagoya 464-8602, Japan}
\email[Y. Terasawa]{yutaka@math.nagoya-u.ac.jp}
\address[Y. Wakasugi]{
Graduate School of Engineering,
Hiroshima University,
Higashi-Hiroshima, 739-8527, Japan}
\email[Y. Wakasugi]{wakasugi@hiroshima-u.ac.jp}

\begin{document}
\begin{abstract}
We study an asymptotic behavior of solutions
to elliptic equations of the second order in a two dimensional exterior domain.
Under the assumption that the solution belongs to $L^q$ with $q \in [2,\infty)$,
we prove a pointwise asymptotic estimate of the solution at the spatial infinity
in terms of the behavior of the coefficients.
As a corollary, we obtain the Liouville-type theorem 
in the case when the coefficients may grow at the spacial infinity. 
We also study a corresponding parabolic problem 
in the $n$-dimensional whole space
and discuss the energy identity
for solutions in $L^q$.
As a corollary we show also the Liouville-type theorem for both  forward and ancient solutions.
\end{abstract}
\keywords{elliptic and parabolic equations of second order; asymptotic behavior}

\maketitle
\section{Introduction}
\footnote[0]{2010 Mathematics Subject Classification. 35J15, 35K10, 35B53 }

We consider an elliptic differential equation of the second order with the 
divergence form such as 
\begin{align}
\label{ell}
	- \sum_{i,j=1}^2 \partial_i (a_{ij}(x) \partial_j u) + \mathbf{b}(x)\cdot \nabla u + c(x) u = 0,
	\quad x \in \Omega
\end{align}
where
$\Omega$
is the whole plane $\mathbb{R}^2$
or an exterior domain
$\Omega = \overline{B_{r_0}(0)}^c = \{ x \in \mathbb{R}^2 ; |x| = r \ge r_0 \}$.
Our aim is to clarify how the asymptotic behavior of the coefficients
$a_{ij}(x)$ and $\mathbf{b}(x)$,
in particular,
their {\it growth} conditions at infinity, has an 
influence to the asymptotic behavior of solution $u(x)$ of (\ref{ell}) as $|x|\to \infty$.
Our study is motivated by investigation of the asymptotic behavior of solutions to
the stationary Navier-Stokes equations
\begin{align}%
\label{ns}
	\left\{ \begin{array}{l}
		-\Delta v + (v \cdot \nabla)v + \nabla p = 0,\\
		\diver v = 0,
	\end{array} \right.
	\quad x \in \Omega.
\end{align}%
By the pioneer work of Leray \cite{Le},
the existence of solutions $(v,p)$ of (\ref{ns}) with the finite Dirichlet integral
\begin{align}%
\label{d_sol}
	\int_{\Omega} |\nabla v (x) |^2 \,dx < \infty
\end{align}%
had been proved.
Then, Gilbarg-Weinberger \cite{GiWe78},
Amick \cite{Am88},
and Korobkov-Pileckas-Russo \cite{KoPiRu17, KoPiRu18}
studied the asymptotic behavior of solutions satisfying \eqref{d_sol}.
They proved that the solution $v$ to \eqref{ns}--\eqref{d_sol}
converges to a constant vector $v_{\infty}$ uniformly at infinity, i.e., 
\begin{align*}%
	\lim_{r\to \infty} \sup_{\theta \in [0,2\pi]} |v(r,\theta) - v_{\infty}| = 0,
\end{align*}%
where $(r,\theta)$ denotes the polar coordinates.
A basic approach to the analysis of \eqref{ns} is to handle
the vorticity
$\omega = \rot v = \partial_{x_1} v_2 - \partial_{x_2} v_1$
which satisfies the equation
\begin{align}%
\label{vor_eq}
	-\Delta \omega + v\cdot \nabla \omega = 0.
\end{align}% 
In our previous result \cite{KoTeWapr},
we studied the asymptotic behavior of solutions $\omega$ to \eqref{vor_eq}
with the finite {\it generalized} Dirichlet integral
\begin{align}%
\label{g_diri}
	\int_{\Omega} |\nabla v(x)|^q \,dx < \infty
\end{align}%
for some $q \in (2,\infty)$.
Note that \eqref{g_diri} implies $\omega \in L^q(\Omega)$.
Indeed, it is proved in \cite{KoTeWapr} that the vorticity $\omega$ and the gradient $\nabla v$ 
of the velocity behave like  
\begin{align*}%
	|\omega (r,\theta)| = o(r^{-(\frac{1}{q} + \frac{1}{q^2})}),\quad 
	|\nabla v(r,\theta)| = o(r^{-(\frac{1}{q} + \frac{1}{q^2})}\log r) 
	\quad
	\mbox{as $r \to \infty$}, 
\end{align*}%
respectively.
The crucial point is to regard the velocity $v(x)$ as a given coefficient in the equation \eqref{vor_eq}
and to analyze how the asymptotic behavior of $v(x)$ does affect 
that of $\omega$.
In this respect, the problem \eqref{ell} may be regarded as a generalization of \eqref{vor_eq}.
%Moreover, from the viewpoint of \eqref{vor_eq},
%our main interest is the case
%$\diver \mathbf{b} = 0$.

In this paper, we generalize the result of \cite{KoTeWapr} to
the elliptic equation \eqref{ell},
and prove the asymptotic behavior of solutions 
at the spatial infinity under the assumption that 
$u \in L^q(\Omega)$ with some $q \in [2,\infty)$.
In particular, we are interested in the case when 
the coefficients $a_{ij}, b$ may grow at spatial infinity.
\par
\vspace{2mm}
Our precise assumptions and results are the following.
\par
\vspace{2mm}
\noindent
{\bf Assumptions on the coefficients for the elliptic problem \eqref{ell}}
\begin{itemize}
\item[(e-i)]
$a_{ij} \in C^1(\Omega)$, $a_{ij}= a_{ji}$ for $i, j =1, 2$  and 
\begin{align*}
	\sum_{i,j=1}^2 a_{ij}(x) \xi_i \xi_j \ge \lambda |\xi|^2 \quad (\xi \in \mathbb{R}^2, x\in \Omega)
\end{align*}
with some $\lambda>0$. The growth condition
\begin{align*}%
	| a_{ij}(x) | = O(|x|^{\alpha}), \quad
	|\partial_i a_{ij}(x) | = O(|x|^{\alpha-1}) \quad (|x|\to \infty)
\end{align*}%
is satisfied for $i,j=1,2$
with some $\alpha \in [0,2]$.
\item[(e-ii)]
$\mathbf{b}(x) = (b_1(x), b_2(x)) \in C^1(\Omega)$
satisfies
\begin{align*}
	\mathbf{b}(x) = O(|x|^{\beta}) \quad (|x| \to \infty)
\end{align*}
with some $\beta \le 1$.
\item[(e-iii)]
$c(x)$ is measurable and {\it nonnegative}.
\item[(e-iv)]
Either following condition (1) or (2) holds:
\begin{itemize}
\item[(1)]
$\diver \mathbf{b}(x) \le 2 c(x)$, 
\item[(2)]
$|\diver \mathbf{b}(x)| = O(|x|^{\beta-1}) \quad (|x| \to \infty)$.
\end{itemize}
\end{itemize} 
\par
\bigskip
Our first result on the elliptic equation (\ref{ell}) now reads
%%%%%%%%%%%%%%%%%%%%%%%%%%%%%%%%%%%%%%%%%%%%%%%%%%%%%%
\begin{theorem}\label{thm_asym}
Let the assumptions {\rm (e-i)--(e-iv)} above hold. 
Suppose that that $u \in C^2(\Omega)$ satisfies \eqref{ell} in $\Omega$ and 
that $u \in L^q (\Omega)$ with some $q \in [2,\infty)$.
Then, we have that 
\begin{equation}\label{decay}
	\sup_{0\le \theta \le 2\pi} |u(r,\theta)| = o( r^{-\frac{1}{q}(1 + \frac{\gamma}{2})} )
	\quad\mbox{as $r \to \infty$},
\end{equation}
where
$\gamma = \min\{1-\beta, 2-\alpha \}$.
\end{theorem}
%%%%%%%%%%%%%%%%%%%%%%%%%%%%%%%%%%%%%%%%%%%%%%%%%%%%%%
\begin{remark}
When
$\alpha = 0$ and $\beta \le -1$, we have
\begin{align*}%
	\sup_{0\le \theta \le 2\pi} |u(r,\theta)| = o( r^{-\frac{2}{q}} )
	\quad\mbox{as $r \to \infty$},
\end{align*}%
which exhibits the correspondence to the condition
$u \in L^q(\Omega)$.
\end{remark}

%\begin{remark}
%In \cite[Theorem 1.1]{KoTeWapr},
%the stationary Navier-Stokes equation in 2D exterior domain
%\begin{align}%
%\label{ns}
%	\left\{ \begin{array}{l}
%		-\Delta v + (v\cdot \nabla )v + \nabla p = 0,\\
%		\diver v = 0
%		\end{array} \right.
%\end{align}%
%is studied.
%Through the analysis of the vorticity equation
%\begin{align*}%
%	-\Delta \omega + v\cdot \nabla \omega = 0,\quad \omega = \rot v,
%\end{align*}%
%the asymptotic behavior of the vorticity
%$\omega = \rot v$
%\begin{align*}%
%	\sup_{\theta \in [0,2\pi]} |\omega(r,\theta)| = o(r^{-(1/q + 1/q^2)})
%	\quad (r \to \infty)
%\end{align*}%
%is proved under the assumption $\nabla v \in L^q$ with some $q \in (2,\infty)$.
%Since $\nabla v \in L^q$ implies $v = O(|x|^{1-2/q})$ as $|x| \to \infty$,
%Theorem \ref{thm_asym} is a generalization of this result to
%more general second order elliptic equation with growing coefficients.
%\end{remark}

As a corollary of the above theorem,
we have the following Liouville-type result.

\begin{corollary}\label{cor_liouville}
Assume {\rm (e-i)--(e-iv)}.
Let $\Omega = \mathbb{R}^2$ and let 
$u \in C^2(\mathbb{R}^2)$ be a solution to \eqref{ell} satisfying
$u \in L^q (\mathbb{R}^2)$
with some $q \in [2,\infty)$.
Then, it holds that $u \equiv 0$ on $\re^2$.
\end{corollary}

\begin{remark}
The above corollary is sharp in the sense that
if $q = \infty$, then there exists a solution $u$ of (\ref{ell}) 
which is not a constant.
Indeed, let 
$a_{ij}(x) = \delta_{ij}$,
$\mathbf{b}(x) = (-x_1, x_2)$ (namely, $\beta = 1$),
and
$c(x) \equiv 0$.
Consider  
$u = u(x_1, x_2) = f(x_1)$ 
with
\begin{align*}%
	f(\tau) = \int_0^{\tau} e^{-s^2/2} \,ds, 
	\quad \tau \in \re.  
\end{align*}%  
It is easy to see that $u \in L^{\infty}(\mathbb{R}^2)$ with 
$- \Delta u + \mathbf{b}(x)\cdot\nabla u = 0$ in $\re^2$.  
Obviously, $u$ is not a constant. 
\end{remark}

\begin{remark}
{\rm 
For the elliptic equation \eqref{ell} with $\Omega = \mathbb{R}^n$ and $n\ge 1$,
the result by
Seregin-Silvestre-\v{S}ver\'{a}k-Zlato\v{s} \cite[Theorem 1.2]{SeSiSvZl12}
implies a Liouville-type theorem under the conditions
that $a_{ij}$ is bounded, $\mathbf{b} \in BMO^{-1}$ with $\diver \mathbf{b} = 0$, 
and that $c \equiv 0$,
namely, every bounded solutions are constants.
Compared with their theorem, our result allows
the coefficients $a_{ij}, b$ to grow at spacial infinity.  
On the other hand, we impose 
on the stronger assumption on the solution such as $u \in L^q(\Omega)$ with 
some $q \in [2,\infty)$
}
\end{remark}

\begin{remark}
{\rm 
The above corollary may be regarded as a generalization of \cite[Corollary 1.2]{KoTeWapr},
which states that
every smooth solution $v$ of \eqref{ns} in $\mathbb{R}^2$
satisfying the condition $\nabla v \in L^q(\mathbb{R}^2)$ for some $q \in (2,\infty)$
must be a constant vector.
Recently, Liouville-type theorems of the stationary Navier-Stokes equations
are fully studied, and we refer the reader to
\cite{BiFuZh13, Ch14, Ch15, ChWo16, ChJaLepr, KoTeWa17, Se16} and the references therein. 
}
\end{remark}

The proofs of Theorem \ref{thm_asym} and Corollary \ref{cor_liouville}
are given in the next section.
Our approach is based on that of Gilbarg and Weinberger \cite{GiWe78}
and its generalization introduced in \cite{KoTeWapr}.
We first show a certain elliptic estimate of the solution $u$ by the energy method
(Lemma 2.1).
Then, combining it with
the integral mean value theorem for the radial variable $r$
and the fundamental theorem of calculus for the angular variable $\theta$,
we derive a pointwise decay estimate of the solution along with
a special sequence $\{ r_n \}_{n=1}^{\infty}$ of the radial variable satisfying 
$\displaystyle{\lim_{n\to\infty}r_n = \infty}$. 
Finally, applying the maximum principle in the annular domains
between $r=r_{n-1}$ and $r=r_n$,
we have the desired uniform decay like (\ref{decay})(Lemma 2.2).

Furthermore, our approach is also applicable to parabolic problems.
We discuss energy estimates of solutions to
the corresponding parabolic equation in the 
$n$-dimensional whole space $\mathbb{R}^n$; 
\begin{align}%
\label{para}
		\partial_t u - \sum_{i,j=1}^n \partial_i (a_{ij}(x,t) \partial_j u)
			+ \mathbf{b}(x,t)\cdot \nabla u + c(x,t) u = 0,
		\quad x \in \mathbb{R}^n, t\in I,
\end{align}%
where
$I \subset \mathbb{R}$ is an interval.
We impose similar assumptions on the coefficients
$a_{ij}(x,t)$, $\mathbf{b}(x,t)$ and $c(x,t)$, on the premise that they are measurable functions 
on $\mathbb{R}^n\times I$:
\par
\noindent
\vspace{2mm}
{\bf Assumptions on the coefficients for the parabolic problem \eqref{para}}
\begin{itemize}
\item[(p-i)]
$a_{ij} \in C^{1,0}(\mathbb{R}^n\times I)$, $a_{ij}=a_{ji}$ for $i, j = 1, \cdots,n$ and 
\begin{align*}
	\sum_{i,j=1}^n a_{ij}(x,t) \xi_i \xi_j \ge \lambda |\xi|^2 \quad (\xi \in \mathbb{R}^n, x\in \mathbb{R}^n, t \in I)
\end{align*}
with some $\lambda>0$. The growth condition 
\begin{align*}%
	| a_{ij}(x,t) | = O(|x|^2),
	\quad |\partial_i a_{ij}(x,t) | = O(|x|) \quad (|x| \to \infty)
\end{align*}%
holds locally uniformly in $t \in I$.
\item[(p-ii)]
$\mathbf{b}(x,t) = (b_1(x,t), \ldots, b_n(x,t)) \in C^{1,0}(\mathbb{R}^n\times I)$
satisfies
\begin{align*}
	\mathbf{b}(x,t) = O(|x|) \quad (|x| \to \infty)
\end{align*}
locally uniformly in $t \in I$.
\item[(p-iii)] 
$c(x, t)$ is nonnegative,  and 
$\diver \mathbf{b}(x,t) = \sum_{j=1}^{n} \partial_{x_j} b_j(x,t) \le 2c(x,t)$
holds for all
$(x,t) \in \mathbb{R}^n \times I$.
\end{itemize}

Under these assumptions,
we show the following energy identity
for solutions belonging to $L^q(\mathbb{R}^n \times I)$.

%%%%%%%%%%%%%%%%%%%%%%%%%%%%%%%%%%%%%%%%%%%%%%%%%5
\begin{theorem}\label{thm_en_est}
Assume {\rm (p-i)} and {\rm (p-ii)}.
Let $u \in C^{2,1}(\mathbb{R}^n \times I)$ be a solution to \eqref{para} satisfying
$u \in L^{q} (\mathbb{R}^n \times I)$
with some $q \in [2,\infty)$.
Then, we have the energy identity
\begin{align}%
\label{en_id}
	&\int_{\mathbb{R}^n} |u(x,t)|^q \,dx
		+ q(q-1) \int_{s}^t \int_{\mathbb{R}^n} |u(x,\tau)|^{q-2} 
		\sum_{i, j =1}^na_{ij}(x, \tau)\partial_iu(x, \tau)\partial_ju(x, \tau)\,dx d\tau \\
\notag
	&\quad + \int_{s}^t \int_{\mathbb{R}^n} (-\diver \mathbf{b}(x,\tau) + qc(x,\tau)) |u(x,\tau)|^q \,dx d\tau \\
\notag
	&=
		\int_{\mathbb{R}^n} |u(x,s)|^q \,dx 
\end{align}%
for all $t, s \in I$ such that  $s \le t$.  
\end{theorem}
%%%%%%%%%%%%%%%%%%%%%%%%%%%%%%%%%%%%%%%%%%%%%%%%%5
\begin{remark}
In Theorem \ref{thm_en_est}, we do not need the assumption {\rm (p-iii)}.
\end{remark}
By Theorem \ref{thm_en_est},
we have the following Liouville-type results
on solutions of the Cauchy problem of \eqref{para} and on ancient solutions of \eqref{para}.
%%%%%%%%%%%%%%%%%%%%%%%%%%%%%%%%%%%%%%%%%%%%%%%%%5
\begin{corollary}\label{cor_liouville_pr}
In addition to {\rm (p-i)} and {\rm (p-ii)}, assume that {\rm (p-iii)} holds.\\
{\rm (i)} Let $u \in C^2(\mathbb{R}^n \times [0, T))$ be a solution of \eqref{para}
satisfying $u \in L^{q} (\mathbb{R}^n \times [0, T))$
with some $q \in [2,\infty)$.
Moreover, we assume that $u(x,0) \equiv 0$ on $\mathbb{R}^n$.
Then, we have $u \equiv 0$ on $\mathbb{R}^n \times [0, T)$.  \\
{\rm (ii)} Let $u \in C^2(\mathbb{R}^n \times (-\infty, 0))$ be 
an ancient solution of \eqref{para} satisfying 
$u \in L^{q} (\mathbb{R}^n \times (-\infty, 0))$
with some $q \in [2,\infty)$.
Then, we have $u \equiv 0$ on $\mathbb{R}^n \times (-\infty, 0)$. 
\end{corollary}
%%%%%%%%%%%%%%%%%%%%%%%%%%%%%%%%%%%%%%%%%%%%%%%%%5

\begin{remark}
{\rm 
For the heat equation
$\partial_t v - \Delta v =0$
on a complete noncompact Reimannian manifold with the 
nonnegative Ricci curvature,
Souplet--Zhang \cite{SoZh} proved that
any positive ancient (or entire) solution $u$ having the bound
\begin{align*}
	u(x,t) = O(e^{o(d(x)+\sqrt{t})})
	\quad\mbox{as $d(x) \to \infty$}
\end{align*}
must be a constant,
where
$d(x)$ is the distance from a base point.
They also proved that any ancient (or entire) solution $u$ having the bound
\begin{align*}
	u(x,t) = o(d(x) + \sqrt{t})
	\quad\mbox{as $d(x) \to \infty$}
\end{align*}
must be a constant.
Compared with their result, we are able to treat more 
general time-dependent coefficients
which may grow at the spatial infinity. 
On the other hand, we impose solutions $u$ on the stronger assumption that 
$u \in L^{q} (\mathbb{R}^n \times I)$ with some $q \in [2,\infty)$.
}
\end{remark}

\begin{remark}
{\rm 
The Liouville-type theorem for
the non-stationary Navier-Stokes equations
\begin{align*}%
	\left\{ \begin{array}{l}
		\partial_t v -\Delta v + (v \cdot \nabla)v + \nabla p = 0,\\
		\diver v = 0,
	\end{array} \right.
	\quad (x,t) \in \mathbb{R}^n\times I
\end{align*}
has been fully studied, where $I=(0, T)$ or $I=(-\infty, 0)$.  
We refer the reader to
\cite{KoNaSeSv09, Ch11, Gi13}
and the references therein.
}
\end{remark}
%%%%%%%%%%%%%%%%%%%%%%%%%%%%%%%%%%%%%%%%%%%%%%%%%
%%%%%%%%%%%%%%%%%%%%%%%%%%%%%%%%%%%%%%%%%%%%%%%%%%
%%%%%%%%%%%%%%%%%%%%%%%%%%%%%%%%%%%%%%%%%%%%%%%%%%
\section{Proof of Theorem \ref{thm_asym}}
In what follows, we shall denote by
$C$
various constants which may change from line to line.
In particular, we denote by
$C = C(*, . . . , *)$
constants depending only on the quantities appearing in parentheses.
%%%%%%%%%%%%%%%%%%%%%%%%%%%%%%%%%%%%%%%%%%%%%%%%%%%%%
\begin{lemma}\label{lem_1}
Under the assumptions on Theorem \ref{thm_asym},
for every $r_1 > r_0$, we have
\begin{align*}
	\int_{r\ge r_1} r^{\gamma} |u|^{q-2} |\nabla u|^2 \,dx
	\le C(q,r_1) \int_{\Omega} |u|^q \,dx,
\end{align*}
where
$\gamma = \min \{ 1-\beta, 2-\alpha \}$.
\end{lemma}
%%%%%%%%%%%%%%%%%%%%%%%%%%%%%%%%%%%%%%%%%%%%%%%%%%%%%%
\begin{proof}
Let
$\eta = \eta(r) \in C_0^{\infty}(\Omega)$
and let
$h = h(u) \in C^1(\mathbb{R})$
be a piecewise $C^2$ function specified later.
We start with the following identity:
\begin{align*}
	& - \sum_{i=1}^2 \partial_i
		\left[ \eta(r) \sum_{j=1}^2 a_{ij} \partial_j (h(u))
			- \sum_{j=1}^2 a_{ij} (\partial_j \eta) h(u) - \eta(r) h(u) b_i(x) \right] \\
	&= - \eta(r) h''(u) \left( \sum_{i,j=1}^2 a_{ij} \partial_i u \partial_j u \right) \\
	&\quad
		+ h(u) \left[ \sum_{i,j = 1}^2 \partial_j (a_{ij} \partial_i \eta)
		+ \mathbf{b}(x) \cdot \nabla \eta(r) + \eta(r)\diver \mathbf{b} \right]\\
	&\quad - \eta (r) h'(u)
		\left[ \sum_{i,j=1}^2 \partial_j (a_{ij} \partial_i u ) - \mathbf{b} \cdot \nabla u \right]
\end{align*}
Since $u$ satisfies the equation \eqref{ell},
integration of the above identity over $\Omega$ yields 
\begin{align*}
	\int_{\Omega} \eta(r) h''(u) \left( \sum_{i,j=1}^2 a_{ij} \partial_i u \partial_j u \right) \,dx
	&= \int_{\Omega} h(u) \left[ \sum_{i,j = 1}^2 \partial_j (a_{ij} \partial_i \eta)
							+ \mathbf{b}(x) \cdot \nabla \eta(r) \right] \,dx \\
	&\quad + \int_{\Omega} \eta(r) (h(u) \diver \mathbf{b}(x) -  h'(u) c(x) u) \,dx.
\end{align*}
Let
$r_1 > r_0$
and let
$\xi_1 = \xi_1(r) \in C^{\infty}(\Omega)$
be nonnegative, monotone increasing in $r$,
and satisfy $\xi(r) = 1$ for $r \ge r_1$ and $\xi(r) = 0$ for $r \le (r_0 + r_1)/2$.
Let
$\xi_2 = \xi_2(r) \in C_0^{\infty}(B_1(0))$
be nonnegative, monotone decreasing, and satisfy
$\xi_2(r) = 1$ for $r \le 1/2$.
We choose the cut-off function
$\eta(r)$
as
\begin{align*}
	\eta(r) = r^{\gamma} \xi_1(r) \xi_2 \left(\frac{r}{R}\right),
\end{align*}
with the parameter $R\ge 1$,
where
$\gamma = \min\{ 1-\beta, 2-\alpha \}$.
Then, we have
$|\nabla \eta(r)| \le Cr^{\gamma-1}$,
$|\partial_i \partial_j \eta(r) | \le Cr^{\gamma-2}$.
Now, we take
$h(u) = |u|^q$.
Then, it holds that 
$h'(u) = q|u|^{q-2}u$
and
$h''(u) = q(q-1) |u|^{q-2}$.
Therefore, we obtain
\begin{align}
\label{eq_energy}
	&q(q-1) \int_{\Omega} \eta(r) |u|^{q-2} \left( \sum_{i,j=1}^2 a_{ij}(x) \partial_i u \partial_j u \right) \,dx \\
\nonumber
	&= \int_{\Omega} |u|^q
		\left[ \sum_{i,j=1}^2 \partial_i ( a_{ij}(x) \partial_j \eta) + \mathbf{b}(x) \cdot \nabla \eta \right] \,dx \\ \nonumber
		& \mbox{} +  \int_{\Omega} \eta ( \diver \mathbf{b}(x) -q c(x) ) |u|^q \,dx.
\nonumber		
\end{align}
By the assumptions (e-i) and (e-ii), the estimates
\begin{align*}%
	& |\mathbf{b}(x) \cdot \nabla \eta(r) | \le C,\\
	& | a_{ij}(x) \partial_i \partial_j \eta(r) | \le C, \quad
	| \partial_i a_{ij}(x) \partial_j \eta(r) | \le C, 
	\quad i, j=1, 2
\end{align*}%
hold, and hence the first term of RHS of \eqref{eq_energy} is estimated by
$C \int_{\Omega} |u|^q \,dx$.
Furthermore, since  $c(x) \ge 0$,  implied by the assumption (e-iii), 
we have by (e-iv) that  
$\diver \mathbf{b}(x) -q c(x) \le 0$
or
$| \eta \diver \mathbf{b} | \le C$,
and hence,
\begin{align*}
	\int_{\Omega} \eta ( \diver \mathbf{b}(x) -q c(x) ) |u|^q \,dx
		\le C \int_{\Omega} |u|^q \,dx
\end{align*}
holds in both cases.
Thus, we obtain from the above estimates and the assumption (e-i) that 
\begin{align*}
	&\int_{r_1 \le r \le R/2} r^{\gamma} |u|^{q-2} |\nabla u|^2\,dx
	\le C \int_{\Omega} |u|^q \,dx.
\end{align*}
Letting $R\to \infty$, we conclude
\begin{align*}
	\int_{r\ge r_1} r^{\gamma} |u|^{q-2} |\nabla u|^2 \,dx
	\le C \int_{\Omega} |u|^q \,dx.
\end{align*}
This completes the proof of Lemma \ref{lem_1}.  
\end{proof}
%%%%%%%%%%%%%%%%%%%%%%%%%%%%%%%%%%%%%%%%%%%%%%%%%%%%%%
%%%%%%%%%%%%%%%%%%%%%%%%%%%%%%%%%%%%%%%%%%%%%%%%%%%%%%
%%%%%%%%%%%%%%%%%%%%%%%%%%%%%%%%%%%%%%%%%%%%%%%%%%%%%%
\begin{lemma}\label{lem_asym}
Under the assumptions on Theorem \ref{thm_asym}, we have
\begin{align*}
	\lim_{r\to \infty} r^{1+\frac{\gamma}{2}} \sup_{\theta \in [0,2\pi]} |u(r,\theta)|^q = 0.
\end{align*}
\end{lemma}
%%%%%%%%%%%%%%%%%%%%%%%%%%%%%%%%%%%%%%%%%%%%%%%%%%%%%%
\begin{proof}
For each sufficiently large integer $n$, let us introduce the quantity
\begin{align*}
	A_n = \int_{2^n}^{2^{n+1}} \frac{dr}{r} \int_0^{2\pi}
		|u|^{q-2} \left( r^2 |u|^2 + r^{1+ \frac{\gamma}{2}} |u| |\partial_{\theta} u | \right) d\theta.
\end{align*}
Since $|\partial_{\theta}u|\le r|\nabla u|$, 
we have by Lemma \ref{lem_1} and the the Schwarz inequality that 
\begin{equation*}
	A_n \le C \int_{2^n<r< 2^{n+1}} (|u|^q + r^{\gamma} |u|^{q-2} |\nabla u|^2 ) \,dx.   
\end{equation*}
On the other hand,
by the mean value theorem for integration, there exists $r_n \in (2^n, 2^{n+1})$ such that 
\begin{align*}
	A_n &= \log 2 \int_{0}^{2\pi} 
			|u(r_n, \theta)|^{q-2} ( r_{n}^2 |u(r_n, \theta)|^2 + r_n^{1+ \frac{\gamma}{2}} 
			|u(r_n, \theta)| |\partial_{\theta} u(r_n, \theta)| )\,d\theta.  
\end{align*}
Next, we estimate
\begin{align*}
	|u(r_n, \theta)|^q - |u(r_n, \varphi)|^q
	&\le \left| \int_{\varphi}^{\theta} \partial_{\psi} | u(r_n, \psi)|^q \,d\psi \right| \\
		&\le \int_0^{2\pi} q |u(r_n, \psi)|^{q-1} |\partial_{\theta} u(r_n, \psi) | \,d\psi.
\end{align*}
Integrating the above for $\varphi \in [0,2\pi]$, we infer
\begin{align*}
	| u(r_n, \theta)|^q
		&\le C \int_0^{2\pi} |u(r_n, \varphi)|^q \,d\varphi
			+ C \int_0^{2\pi} q |u(r_n, \psi)|^{q-1} |\partial_{\theta} u(r_n, \psi) | \,d\psi.
\end{align*}
Multiplying both sides of this estimate by $r_n^{1 + \frac{\gamma}{2}}$ 
and then noting $1 + \frac{\gamma}{2}\le 2$, we have that 
\begin{align*}
	r_n^{1 + \frac{\gamma}{2}} |u(r_n, \theta)|^q
		&\le C r_n^{1 + \frac{\gamma}{2}} \int_0^{2\pi} |u(r_n, \varphi)|^q \,d\varphi \\
		&\quad + C r_n^{1 + \frac{\gamma}{2}} \int_0^{2\pi} q |u(r_n, \psi)|^{q-1} |\partial_{\theta} u(r_n, \psi) | \,d\psi \\
		&\le C A_n,
\end{align*}
Consequently, we obtain
\begin{align*}
	r_n^{1 + \frac{\gamma}{2}} |u(r_n, \theta)|^q
	\le \int_{r>2^n} (|u|^q + r^{\gamma} |u|^{q-2}|\nabla u|^2)  \,dx. 
\end{align*}
Since the right-hand side of the above inequality tends to zero as $n \to \infty$,
implied by Lemma \ref{lem_1}, we have that 
\begin{equation}\label{eqn:2.1}
	\lim_{n\to \infty} r_n^{1 + \frac{\gamma}{2}} \sup_{\theta\in [0,2\pi]} |u(r_n, \theta)|^q = 0.
\end{equation}
Finally, since the solution $u$ of \eqref{ell} satisfies the maximum principle 
and since $r_{n+1} \le 4 r_n$,
we estimate for $r \in (r_n, r_{n+1})$ that 
\begin{align*}
	& r^{1 + \frac{\gamma}{2}} \sup_{\theta\in [0,2\pi]} |u(r, \theta)|^q \\
	 	&\le  r_{n+1}^{1 + \frac{\gamma}{2}}
			\max\{ \sup_{\theta\in [0,2\pi]} |u(r_{n}, \theta)|^q, \sup_{\theta\in [0,2\pi]} |u(r_{n+1}, \theta)|^q \} \\
		&\le \max \{ 16 r_n^{1 + \frac{\gamma}{2}} \sup_{\theta\in [0,2\pi]} |u(r_{n}, \theta)|^q,
					r_{n+1}^{1 + \frac{\gamma}{2}} \sup_{\theta\in [0,2\pi]} |u(r_{n+1}, \theta)|^q \},
\end{align*}
which yields with the aid of (\ref{eqn:2.1}) that 
\begin{align*}
	\lim_{r\to \infty}r^{1 + \frac{\gamma}{2}} \sup_{\theta\in [0,2\pi]} |u(r, \theta)|^q =0.
\end{align*}
This completes the proof of Lemma \ref{lem_asym}, and whence Theorem \ref{thm_en_est}.  
\end{proof}

\begin{proof}[Proof of Corollary \ref{cor_liouville}]
By the assumption (e-i), the equation \eqref{ell} has the maximum principle.
Combining this with the asymptotic behavior from Theorem \ref{thm_asym},
we have $u\equiv 0$.
\end{proof}%AAA

%%%%%%%%%%%%%%%%%%%%%%%%%%%%%%%%%%%%%%%%%%%%%%%%%%
%%%%%%%%%%%%%%%%%%%%%%%%%%%%%%%%%%%%%%%%%%%%%%%%%%
%%%%%%%%%%%%%%%%%%%%%%%%%%%%%%%%%%%%%%%%%%%%%%%%%%
\section{Proof of Theorem \ref{thm_en_est}}
Let $h(u) = |u|^q$.
We take a nonnegative function
$\psi \in C_0^{\infty}(\mathbb{R}^n)$
such that
\begin{align*}%
	\psi(x) = \begin{cases}
		1 &(|x| \le 1),\\
		0 &(|x| \ge 2),
		\end{cases}
\end{align*}%
and with the parameter $R>0$ we define
\begin{align*}%
	\psi_R (x) = \psi \left( \frac{x}{R} \right).
\end{align*}%
Similarly to the previous section, by a direct computation we have
\begin{align*}
	& - \sum_{i=1}^n \partial_i
		\left[ \psi_R \sum_{i=j}^n a_{ij} \partial_j (h(u))
			- \sum_{j=1}^n a_{ij} (\partial_j \psi_R) h(u) - \psi_R h(u) b_i \right] \\
	&= - \psi_R h''(u) \left( \sum_{i,j=1}^n a_{ij} \partial_i u \partial_j u \right) \\
	&\quad
		+ h(u)\left[\sum_{i,j = 1}^n \partial_i (a_{ij} \partial_j \psi_R)
		+ \mathbf{b} \cdot \nabla \psi_R + \psi_R \diver \mathbf{b}\right] \\
	&\quad - \psi_R h'(u)
		\left[ \sum_{i,j=1}^n \partial_i (a_{ij} \partial_j u ) - \mathbf{b} \cdot \nabla u \right]
\end{align*}
Using the equation \eqref{para} and the identity $h'(u) \partial_t u = \partial_t ( h(u) )$,
we integrate the above identity over $\mathbb{R}^n$ to obtain
\begin{align*}%
	&\frac{d}{dt} \int_{\mathbb{R}^n} \psi_R h(u) \,dx
		+ \int_{\mathbb{R}^n} \psi_R h''(u) \left( \sum_{i,j=1}^2 a_{ij} \partial_i u \partial_j u \right)\,dx \\
	&= \int_{\mathbb{R}^n} h(u)
		\left[ \sum_{i,j=1}^n \partial_i \left( a_{ij} \partial_j \psi_R \right) + \mathbf{b} \cdot \nabla \psi_R \right] \,dx 
		+ \int_{\mathbb{R}^n} \psi_R \left( h(u) \diver \mathbf{b}  - ch'(u) u \right) \,dx.
\end{align*}%
Furthermore, since $h'(u) = q |u|^{q-2}u$ and $h''(u) = q(q-1) |u|^{q-2}$, 
we integrate the above identity over $[s,t]$ to obtain 
\begin{align}%
\label{en_est_1}
	&\int_{\mathbb{R}^n} \psi_R h(u(t)) \,dx
		+ q(q-1) \int_{s}^t \int_{\mathbb{R}^n} 
		\psi_R |u|^{q-2} \sum_{i, j=1}^na_{ij}\partial_iu\partial_ju  \,dx d\tau \\
\notag
	&\quad
		+ \int_{s}^t \int_{\mathbb{R}^n} \psi_R (-\diver \mathbf{b} + q c) |u|^q \,dx d\tau \\
\notag
	&=
		\int_{\mathbb{R}^n} \psi_R h(u(s)) \,dx \\
\notag
	&\quad
		+ \int_{s}^t \int_{\mathbb{R}^n} h(u)
			\left[ \sum_{i,j=1}^n \partial_j \left( a_{ij} \partial_i \psi_R \right)
				+ \mathbf{b} \cdot \nabla \psi_R \right] \,dxd\tau.
\end{align}%
Let us estimate the right-hand side.
Using the assumptions (p-i) and (p-ii), and then applying the Lebesgue dominated convergence theorem,
we have that
\begin{align*}%
	&\int_{s}^t \int_{\mathbb{R}^n} h(u)
		\sum_{i,j=1}^n \partial_j \left( a_{ij} \partial_i \psi_R \right) \,dxd\tau \\
	&\le C R^{-2} \int_{s}^t \int_{B_{2R}\setminus B_R} |u|^q |x|^2 \,dxd\tau
		+ CR^{-1} \int_{s}^t \int_{B_{2R}\setminus B_R} |u|^q |x| \,dxd\tau \\
	&\le C \int_{s}^t \| u (\cdot, \tau) \|_{L^q(B_{2R}\setminus B_R)}^q \,d\tau \\
	&\to 0 \quad (R\to \infty)
\end{align*}%
and
\begin{align*}%
	\int_{s}^t \int_{\mathbb{R}^n} h(u) \mathbf{b} \cdot \nabla \psi_R \,dxd\tau
	&\le CR^{-1} \int_{s}^t \int_{B_{2R}\setminus B_R} |u|^q |x| \,dxd\tau \\
	&\le C \int_{s}^t \| u (\cdot, \tau) \|_{L^q(B_{2R}\setminus B_R)}^q \,d\tau \\
	&\to 0 \quad (R\to \infty).
\end{align*}%
Consequently, letting $R\to \infty$ in \eqref{en_est_1}, we obtain
\begin{align*}%
	&\int_{\mathbb{R}^n} |u(x,t)|^q \,dx
		+ q(q-1) \int_{s}^t \int_{\mathbb{R}^n} |u(x,\tau)|^{q-2} 
		\sum_{i, j =1}^na_{ij}(x, \tau)\partial_iu(x, \tau)\partial_ju(x, \tau)
		\,dx d\tau \\
\notag
	&\quad
		+ \int_{s}^t \int_{\mathbb{R}^n} (- \diver \mathbf{b}(x,\tau) + qc(x,\tau)) |u(x,\tau)|^q \,dx d\tau \\
\notag
	&=
		\int_{\mathbb{R}^n} |u(x,s)|^q \,dx.
\end{align*}%
This completes the proof of Theorem \ref {thm_en_est}.

%%%%%%%%%%%%%%%%%%%%%%%%%%%%%%%%%%%%%%%%%%
\begin{proof}[Proof of Corollary \ref{cor_liouville_pr}]
(i) Applying Theorem \ref{thm_en_est} with $s = 0$ and using $u(x,0) = 0$,
we have by the assumption (p-i) that 
\begin{align*}%
	&\int_{\mathbb{R}^n} |u(x,t)|^q \,dx
		+ q(q-1)\lambda \int_{0}^t \int_{\mathbb{R}^n} |u(x,\tau)|^{q-2} |\nabla u(x,\tau)|^2 \,dx d\tau \\
\notag
	&\quad
		+ \int_{0}^t \int_{\mathbb{R}^n} (- \diver \mathbf{b}(x,\tau) + qc(x,\tau))  |u(x,\tau)|^q \,dx d\tau \\
\notag
	&\le 0
\end{align*}%
for all $t \in [0,T)$.
Noting
$- \diver \mathbf{b} + qc \ge 0$
by the assumption (p-iii),
we conclude $u(x,t) = 0$ for $(x,t) \in \mathbb{R}^n \times [0, T)$.
\par
(ii)
Let $I = (-\infty, 0)$. 
Similarly to the above (i), applying Theorem \ref{thm_en_est} for $s < t < 0$,
we have by the assumption (p-iii) that 
\begin{align}%
\label{est_ulq}
	\| u(t) \|_{L^q} \le \| u(s) \|_{L^q}.
\end{align}%
Since
$u \in L^q(\mathbb{R}^n \times (-\infty, 0))$,
there exists a sequence $\{ s_n \}_{n=1}^{\infty} \subset (-\infty, 0)$
such that
$\lim_{n\to \infty} s_n = - \infty$
and
$\lim_{n\to \infty} \| u(s_n) \|_{L^q} = 0$.
Therefore, taking $s = s_n$ in \eqref{est_ulq} and then letting $n \to \infty$,
we have
$\| u(t) \|_{L^q} = 0$. 
Since $t \in (-\infty, 0)$ is arbitrary,  
we conclude that $u \equiv 0$ on $\mathbb{R}^n \times (-\infty, 0)$. 
This proves Corollary  \ref{cor_liouville_pr}.  
\end{proof}

\section*{Acknowledgement}
This work was supported by JSPS Grant-in-Aid for Scientific Research(S) Grant Number JP16H06339.

\end{document}